\documentclass{amsart}
\usepackage[latin1]{inputenc}
\usepackage[english]{babel}
\usepackage[T1]{fontenc}
\usepackage{amsmath}
\usepackage{amssymb}
\usepackage{amsthm}
\usepackage{array,booktabs}
\usepackage{graphicx}
\usepackage{newtxtext}
\usepackage{quoting}
\usepackage{enumerate}
\usepackage{mathtools}
\usepackage{csquotes}
\usepackage{cite}
\usepackage{tikz}
\usepackage{tikz-cd}
\usetikzlibrary{matrix,arrows,decorations.pathmorphing}
\usepackage{wrapfig}
\renewcommand{\Re}{\operatorname{Re}}
\renewcommand{\Im}{\operatorname{Im}}
\def\R{\ensuremath\mathbb{R}}
\def\C{\ensuremath\mathbb{C}}
\def\Z{\ensuremath\mathbb{Z}}
\def\Q{\ensuremath\mathbb{Q}}

\def\H{\ensuremath\mathbb{H}}

\usepackage{hyperref}
\usepackage[final]{pdfpages}
\usepackage{verbatim}
\newtheorem{thm}{Theorem}[section]

\newtheorem{lemma}[thm]{Lemma}
\newtheorem{prop}[thm]{Proposition}
\theoremstyle{remark}
\newtheorem{remark}[thm]{Remark}

\def\eps{\ensuremath\varepsilon}
\def\0{\emptyset}

\def\Sym{\hbox{\rm Sym}}
\def\PSL{\hbox{\rm PSL}}

\def\vol{\hbox{\rm vol}}
\numberwithin{equation}{section}
\begin{document}
\title[Distribution of periods of holomorphic cusp forms]{On the distribution of periods of holomorphic cusp forms and zeroes of period polynomials}
\author{Asbjørn Christian Nordentoft}
\address{Department of Mathematical Sciences, Copenhagen University, Universitetsparken 5, Copenhagen 2100, Denmark}
\email{\href{mailto:acnordentoft@outlook.com}{acnordentoft@outlook.com}}
\date{\today}
\subjclass[2010]{11F67(primary), and 11L05(secondary)}
\begin{abstract}
In this paper we determine the limiting distribution of the image of the Eichler--Shimura map or equivalently the limiting joint distribution of the coefficients of the period polynomials associated to a fixed cusp form. The limiting distribution is shown to be the distribution of a certain transformation of two independent random variables both of which are equidistributed on the circle $\R/\Z$, where the transformation is connected to the additive twist of the cuspidal $L$-function. Furthermore we determine the asymptotic behavior of the zeroes of the period polynomials of a fixed cusp form. We use the method of moments and the main ingredients in the proofs are additive twists of $L$-functions and bounds for both individual and sums of Kloosterman sums.     
\end{abstract}
\maketitle
\section{Introduction}
Understanding the special values of $L$-functions is a notoriously hard problem and has deep arithmetic content due to the conjectures of Birch--Swinnerton-Dyer and Bloch--Kato. As a striking example of the connection between $L$-functions and arithmetics, Kolyvagin \cite{Kolyvagin88} proved that if $E/\Q$ is an elliptic curve such that the central value of the Hasse--Weil zeta function $L(E,1)$ is {\it non-zero}, then the set of rational points $E(\Q)$ is {\it finite}.\\ 
Periods of automorphic forms have been an indispensable tool in the study of $L$-functions since the beginning of the theory (Hecke, Rankin--Selberg, Shimura, Manin) and continue to be so to this day (\cite{Ichino08}, \cite{SakVenk17}, \cite{NelVenk18}). This paper is concerned with the distribution properties of automorphic periods, which in many cases are much more well-behaved and easier to handle than the values of the $L$-functions themselves (see below for a toy example of this phenomena).\\ 

The {\it Eichler--Shimura map} defines an isomorphism between the space of weight $k$ holomorphic cusp forms and a parabolic cohomology group introduced by Eichler, by sending a cusp form to its periods. Our main result (see Theorem \ref{distN}) describes the asymptotic distribution of the periods of a fixed cusp form or equivalently the asymptotic joint distribution of the coefficients of period polynomials. Furthermore we use our methods to derive an asymptotic expression for the zeroes of the period polynomials of a fixed cusp form (see Theorem \ref{zeroes}), supplementing recent work of Jin, Ma, Ono, and Soundararajan \cite{JiMaOnSo16}, see also \cite{DiRo18}.  \\
For $k=2$ the period polynomials degenerate to constants and are known as {\it modular symbols} introduced by Birch and Manin. Petridis and Risager \cite{PeRi2}, \cite{PeRi} showed that modular symbols appropriately ordered are asymptotically normally distributed. From a cohomological point of view, the period polynomials are the natural generalization of modular symbols, but in this paper we show however that for $k\geq 4$ the coefficients of the period polynomials behave very differently from modular symbols.
\subsection{A toy example}
To illustrate the relation between periods and $L$-functions, let us consider a toy example of such automorphic periods given by the rational values of the complex exponential; $e^{2\pi i r}, r\in \Q$. These periods are connected to {\it Gauss sums}, which will serve as analogues of $L$-functions in this discussion;  
$$ \tau(\chi):=\sum_{a\in(\Z/q\Z)^\times } \chi(a) e^{2\pi i a/q},  $$ 
where $q$ is a positive integer and $\chi: (\Z/q\Z)^\times\rightarrow \C$ is a primitive Dirichlet character (see (\ref{periodformula}) below for one possible justification for the analogy between Gauss sums and $L$-functions). Gauss sums are in fact intimately connected to $L$-functions since $i^{-\kappa}\tau(\chi)q^{-1/2}$ is the {\it root number} of the Dirichlet $L$-function $L(\chi,s)$, where $\kappa=\tfrac{1-\chi(-1)}{2}$. More precisely the functional equation for Dirichlet $L$-functions takes the form;
$$  \Lambda(\chi,s):=\Gamma\left(\frac{s+\kappa}{2}\right)\left(\frac{q}{\pi}\right)^{s/2} L(\chi,s)=\frac{i^{-\kappa}\tau(\chi)}{q^{1/2}}\Lambda(\overline{\chi},1-s) .   $$

To illustrate the difference in difficulty between dealing with periods and $L$-functions, we will consider the problem of determining the distribution of respectively the rational values of the complex exponential and the Gauss sums. It is easy to show that the periods themselves; 
$$P_q:=\{e^{2\pi i a/q}\mid (a,q)=1\}$$ 
equidistribute on the unit circle as $q\rightarrow \infty$ (notice that this is not completely trivial because of the co-primality condition). In this case the Weyl sums for the distribution problem are Ramanujan sums, which can be evaluated explicitly. \\
On the other hand Gauss showed that $\tau(\chi)$ always has absolute value equal to $q^{1/2}$. But understanding the value distribution of 
$$L_q:=\{\tau(\chi)\mid \chi:(\Z/q\Z)^\times\rightarrow \C,\text{primitive Dirichlet character}\},$$ 
as $q\rightarrow \infty$ turned out to be a much more difficult problem. This problem was solved by Katz \cite{Katz80} who showed (for $q$ prime) that the Gauss sums also equidistribute (now on the circle with radius $q^{1/2}$) using deep input from algebraic geometry.\\
This example illustrates in a very simple setting the difference in difficulty between dealing with automorphic periods and $L$-functions themselves. 

\subsection{The periods of holomorphic cusp forms}
In this paper we study periods of holomorphic cusp forms. The most famous example of a cusp form is probably the modular $\Delta$-function introduced by Ramanujan as the following $q$-series;
$$ \Delta(z):=q\prod_{n\geq 1}(1-q^n)^{24}=\tau(1)q+\tau(2)q^2+\tau(3)q^3+\ldots,\quad q=e^{2\pi i z}.$$
In this case, given a primitive Dirichlet character $\chi:(\Z/q\Z)^\times\rightarrow \C$, we define the twisted $L$-function;
$$ L(\Delta, \chi, s):= \sum_{n\geq 1} \chi(n)\tau(n)n^{-s}, $$
which converges absolutely for $\Re s>13/2$ and admits analytic continuation with a functional equation relating $s\leftrightarrow 12-s$. In this case the special values $s=1,\ldots, 11$ can be written as a twisted rational linear combination of {\it the periods of $\Delta$};
\begin{align}\label{periodformula}\pi^{-m}L(\Delta, \chi, m)=\sum_{\substack{-q/2<a<q/2,\\ 0\leq l\leq 10}} c(a/q,l,m)\, \chi(a) \underbrace{\int_{a/q}^{i\infty} \Delta(z) z^ldz}_{\rm periods},\end{align}
where $m\in \{1,\ldots, 11\}$ and $c(a/q,l,m)\in \Q$ (see \cite{Manin73} for details, where this is used to prove rationality results for $L(\Delta, \chi, m)$ and to construct $p$-adic $L$-functions). Notice the similarity between this formula for the twisted special values and the formula for Gauss sums in the toy example above. \\
We will study the distribution of the periods of holomorphic cusp forms appearing in (\ref{periodformula}) or equivalently of the image of the Eichler--Shimura map.\\

To be more precise let $ \mathcal{S}_k(\Gamma_0(N))$ denote the space of cusp forms of even weight $k$ and level $N$. To each cusp form $f\in \mathcal{S}_k(\Gamma_0(N))$ and each $\gamma \in \Gamma_0(N),$ the Eichler--Shimura map associates the following $(k-1)$-dimensional complex vector consisting of the periods of $f$;
\begin{align}\label{defperiod} &u_f(\gamma)=(u_{f,0}(\gamma), u_{f,1}(\gamma), \ldots, u_{f,k-2}(\gamma))\\
\nonumber &:= \Biggr(\int_{\gamma \infty}^\infty f(z) dz ,\int_{\gamma \infty}^\infty f(z) z dz,\ldots , \int_{\gamma \infty}^\infty f(z) z^{k-2}dz\Biggr)^T\in \C^{k-1} ,\end{align}
where $^T$ denotes matrix transpose and $\gamma\infty=a/c$ with $a,c$ the left-upper and -lower entry of $\gamma$. The map $u_f:\Gamma_0(N) \rightarrow \C^{k-1}$ can be shown to satisfy a 1-cocycle relation with respect to a certain action of $\Gamma_0(N)$ on $\C^{k-1}$, which we will make precise below in Section \ref{EichlerShimura}. Thus $u_f$ defines an element of the cohomology group $H^1(\Gamma_0(N), M)$, where $M$ is given by $\C^{k-1}$ equipped with the just mentioned action of $\Gamma_0(N)$. The association $f\mapsto u_f$ is a constituent of the {\it Eichler--Shimura isomorphism} as we will see below.\\
When ordered by the denominator of the cusp $\gamma \infty$, we show that the limiting distribution of $u_f(\gamma)$ is the distribution of a certain transformation of two independent random variables both of which are uniformly distributed on the circle $\R/\Z$ (see Theorem \ref{distN} and Theorem \ref{distN2} below for the precise statements).
\subsection{Results for $\Gamma_0(N)$} 
Let $f\in \mathcal{S}_k(\Gamma_0(N))$ be a cusp form of weight $k$ with Fourier expansion;
$$ f(z)=\sum_{n\geq 1} a_f(n) q^n,\quad q=e^{2\pi i z}.$$
Then for each $x\in \R$, we define the following Dirichlet series called the {\it additive twist by $x$} of the $L$-function of $f$;
\begin{align}  \label{additivetwist11}L(f,x ,s):=\sum_{n\geq 1} \frac{a_f(n)e(nx)}{n^s}, \end{align}
where $e(x)=e^{2\pi i x}$. This Dirichlet series converges absolutely for any $x\in \R$ when $\Re s>(k+1)/2$ by Hecke's bound; 
\begin{align}\label{hecke}\sum_{n\leq X}|a_f(n)|^2\ll_f X^k,\end{align}
which is known to hold for cusp forms for general Fuchsian groups of the first kind. When $x$ corresponds to a cusp (i.e. $x\in\Q$), the additive twist by $x$ satisfies analytic continuation to the entire complex plane and if $x$ is $\Gamma_0(N)$ equivalent to $\infty$, we also have a functional equation relating $s$ and $k-s$ (see Section \ref{AT-section} for details).\\
For $c>0$ such that $N|c$, we consider the periods $u_f$ as a $(k-1)$-dimensional complex random variable defined on the outcome space;
\begin{align}\label{Tc}\Omega_c:=\{ a/c\in \Q \mid  a,c \in \Z_{\geq 0}, (a,c)=1, 0\leq a<c \},\end{align}
endowed with the uniform probability measure, where $u_f(a/c):=u_f(\gamma)$ for $a/c=\gamma \infty$ (i.e. $a,c$ are the left upper- and lower entries of $\gamma\in \Gamma_0(N))$.\\
Our main result is that the limiting distribution as $c\rightarrow \infty$ (when appropriately normalized) is the tranformation of two independent distributions on the circle. 
\begin{thm}\label{distN}
Let $ f\in \mathcal{S}_k(\Gamma_0(N))$ be a cusp form of even weight $k\geq 4$ and level $N$. Then we have for any fixed box $A \subset \C^{k-1}$ that
\begin{align}
\nonumber \mathbb{P}_{\Omega_c}\left( \frac{u_f(a/c)}{C_k c^{k-2}}\in A \right):=& \frac{\#\{\frac{a}{c}\in \Omega_c\mid \frac{u_f(a/c)}{C_k c^{k-2}}\in A  \} }{\# \Omega_c}\\
 = &\mathbb{P}\left( F(Y,Z)\in A \right)+o(1)
\end{align}
as $c\rightarrow \infty$ with $N|c$, where $Y,Z$ are two independent random variables both distributed uniformly on $[0,1)$, $F: [0,1)\times [0,1) \rightarrow \C^{k-1}$ is given by
$$  F(y,z):= L(f,y , k-1)\left(1, z, \ldots, z^{k-2}\right)^T,  $$
and $C_k=\frac{i\Gamma(k-1)}{(2\pi )^{k-2}}$.\\ 
(Here $\mathbb{P}\left( F(Y,Z)\in A \right)$ denotes the probability of the event $ F(Y,Z)\in A$).
\end{thm}
\begin{remark}
As was noted in \cite[Section 1.4.1]{BeDr19} the individual distribution of the critical values of $L(f,\gamma \infty,s)$ for $s\neq k/2$ are not that interesting since for $\Re s>(k+1)/2$ the critical values are rational values of a continuous function and consequently the limiting distribution is just the pullback by this continuous function of the Lebesgue measure on the circle $\R/\Z$, since reduced fractions equidistribute (and similarly for $\Re s<(k+1)/2$ using the functional equation). In order to handle the distribution of the Eichler--Shimura map (or equivalently the coefficients of period polynomials), we however need to control the dependence between the different critical values of $L(f,\gamma \infty,s)$ and maps of the type $\gamma \mapsto (\gamma \infty)^j$. In the end, the specific shape of the limiting distribution amounts to the non-trivial cancellation in sum of Kloosterman sums with uniformity in the frequencies and thus non-trivial input is needed. 
\end{remark}
\begin{remark}
Given an orthogonal basis $f_1,\ldots, f_d$ for $\mathcal{S}_k(\Gamma_0(N))$, we can also compute the joint distribution of 
$$u_{k,N}:=(u_{f_1}, \ldots, u_{f_d})^T\in \C^{d(k-1)},$$ 
when appropriately normalized, with a similar proof. We have however restricted the exposition to a single cusp form $f$ for notational simplicity. For the complete orthogonal basis the result is that the random variables defined from $\frac{(2\pi  /c )^{k-2}}{i\Gamma(k-1)}u_{k,N}$ converge in distribution (in the same sense as in Theorem \ref{distN} above) to the random variable 
$$F_{k,N}(Y,Z),$$
where $Y,Z$ are two independent and uniformly distributed random variables on $[0,1)$ and $F_{k,N}:[0,1)\times[0,1)\rightarrow \C^{d(k-1)}$ is given by 
\begin{align*}F_{k,N}(y,z):= &\Biggr( L(f_1,Y , k-1),\ldots ,  L(f_1,Y , k-1)z^{k-2}, \\
&\qquad \qquad \qquad \qquad \ldots, L(f_d,Y , k-1)z^{k-2}\Biggr)^T\in  \C^{d(k-1)}.  \end{align*}
In particular it is worth noticing that $u_{f_i}(\gamma)$ and $u_{f_j}(\gamma)$ for $i\neq j$ are highly dependent as opposed to the case $k=2$ (see \cite[Theorem 5.1]{No19}).
\end{remark}
\begin{remark}\label{Ono}
If $f\in \mathcal{S}_k(\Gamma_0(N))$ then it follows from work of Jin, Ma, Ono and Soundararajan \cite[Theorem 1.2]{JiMaOnSo16} that for $k\geq 6$ the period polynomials $r_{f,S}(\sqrt{N}X)$ (see (\ref{periodpoldef}) for a definition) converge coefficient for coefficient to $X^{k-2}-1$ as $N\rightarrow \infty$. 
\end{remark}
\begin{remark}The author \cite{No19} and independently Bettin and Drappeau \cite{BeDr19} (for level 1) have considered the distribution of central values of additive twists of $L$-functions of cusp forms of arbitrary even weight and showed that they are normally distributed. As was also noted in \cite[Section 3.3.2]{No19} the coefficients of the period polynomial can be expressed as linear combinations of critical values of additive twists (including the central value). However the left-most critical value at $s=1$ will be the dominating term, which is why we see that the distribution degenerates (and in particular is not normal). \end{remark}

\subsection{Zeroes of period polynomials} 
The vector $u_f$ encodes the periods of $f\in \mathcal{S}_k(\Gamma_0(N))$, which were introduced in a slightly different setting by M. Eichler in his study of parabolic cohomology \cite{Ei59}. He defined the {\it period polynomials} associated to $f$ as 
\begin{align}\label{periodpoldef} r_{f,\gamma}(X):&=\int_{\gamma\infty}^{\infty} f(z)(z-X)^{k-2}dz\\
\nonumber &= \sum_{j=0}^{k-2} X^j (-1)^j\binom{k-2}{j}\int_{\gamma\infty}^\infty f(z)z^{k-2-j}dz, \end{align}
where $\gamma\in \Gamma_0(N)$. Note that the periods of $f$ are equal to the coefficients of this polynomial (up to a scaling by factorials). The Eichler--Shimura isomorphism can also be described intrinsically and naturally in terms of period polynomials as was done in \cite{PaPo}. Our results can be interpreted as determining the joint distribution of the coefficients of the period polynomials.\\ 
Recently there has been a lot of study in the analytic properties of period polynomials, especially the location of the zeroes of $r_{f, S}$, where $S=\begin{psmallmatrix} 0 & -1 \\ 1 & 0 \end{psmallmatrix}$ (see \cite{DiRo18} for a complete list of references). The results of this paper should be seen more in relation with these results rather than with those of Petridis and Risager \cite{PeRi}. \\

For $f\in \mathcal{S}_k(\Gamma_0(N))$ a newform of even weight $k\geq 6$, we can use our methods to understand the zeroes of $r_{f,a/c}$ asymptotically as $c\rightarrow \infty$. The assumptions on $f$ are made in order to ensure that $L(f,x ,k-1)$ is non-zero for all $x\in \R$.
\begin{thm}\label{zeroes} Let $f\in \mathcal{S}_k(\Gamma_0(N))$ be a newform of even weight $k\geq 6$ and level $N$. Then $r_{f, \gamma}$ is a polynomial of degree $k-2$ for any $\gamma\in \Gamma_0(N)$. Furthermore all zeroes $x_0$ of $r_{f,\gamma}$ satisfy 
$$x_0= a/c+O_{k}((|a/c|+1)^{(k-4)/(k-2)}c^{-2/(k-2)}),$$
where $a,c$ are the entries in the left column of $\gamma$ (i.e. $\gamma \infty=a/c$).
\end{thm}
\begin{remark}
Analogously Jin, Ma, Ono and Soundararajan \cite[Theorem 1.2]{JiMaOnSo16} building on works of others (see \cite{DiRo18}) determined the zeroes of $r_{f,S}$ as either the weight $k$ or level $N$ tend to infinity. In their case the zeroes satisfy a version of the Riemann Hypothesis, of which no analogue seems to exist in our setting. 
\end{remark}
\subsection{Results for general cofinite Fuchsian groups} 
We also obtain results for a general cofinite, discrete subgroup $\Gamma$ of $\PSL_2(\R)$ with a cusp at infinity of width 1 (see \cite[Chapter 2]{Iw} for definitions), but we have to take an extra average. Given a cusp form $f\in \mathcal{S}_k(\Gamma)$, we can similarly define the additive twists $L(f,x ,s)$ of the associated $L$-function, which satisfy the same properties as in the case of Hecke congruence groups, as we will explain in Section \ref{AT-section} below.\\
To state our results, we introduce the following set;
\begin{align}\label{T}T_{\leq 1}=T_{\leq 1, \Gamma}:=\left\{ r=\gamma \infty \in \R\mid \gamma\in \Gamma / \Gamma_\infty, 0\leq r<1 \right\}. \end{align} 
This is a slight modification of the set $T=T_\Gamma$ defined in \cite{PeRi}, which parametrizes the double coset $\Gamma_\infty \backslash \Gamma / \Gamma_\infty$. In this paper we need to choose a representative, since $u_f(\gamma)$ is not invariant under the action of $\Gamma_\infty$ from the left. One would get similar results by choosing different representatives.\\ 
Using the argument in the proof of \cite[Proposition 2.2]{PeRi}, we see that to any $r\in T_{\leq 1}$ there is a unique $\gamma \in \Gamma / \Gamma_\infty$ with lower-left entry $c>0$ such that $r=\gamma \infty$ and we define $c(r):=c$. Now for $X>0$, we consider $u_f$ as a random variable on the outcome space;
\begin{align}\label{OmegaX}\tilde{\Omega}_X:= \{r\in T_{\leq 1}\mid c(r)\leq X \}.\end{align}
endowed with the uniform probability measure. In this setting our result is the following. 
 \begin{thm}\label{distN2}
Let $ f\in \mathcal{S}_k(\Gamma)$ be a cusp form of even weight $k\geq 4$. Then we have for any fixed box $A \subset \C^{k-1}$ that
\begin{align}
\nonumber \mathbb{P}_{\tilde{\Omega}_X}\left( \frac{u_f(r)}{C_k c(r)^{k-2}}\in A \right):=& \frac{\#\{r\in \tilde{\Omega}_X\mid \frac{u_f(r)}{C_k c(r)^{k-2}}\in A  \} }{\# \tilde{\Omega}_X}\\ 
=&\mathbb{P}\left( F(Y,Z)\in A \right)+o(1)
\end{align}
as $X\rightarrow \infty$, where $Y,Z$ are two independent random variables both distributed uniformly on $[0,1)$, $F: [0,1)\times [0,1) \rightarrow \C^{k-1} $ and $C_k$ as in Theorem \ref{distN}.
\end{thm}
\section*{Acknowledgement}
I would like to express my gratitude to Dorian Goldfeld and Columbia University for their hospitality and to my advisor Morten Risager and Riccardo Pengo for valuable suggestions. Finally I would like to thank the referees for their careful readings, which improved the quality of the paper.

\section{Preliminaries and Background}
In this section we will introduce some background on respectively the Eichler--Shimura isomorphism, bounds on sums of Kloosterman sums and finally additive twists of modular $L$-functions. 
\subsection{Background on the Eichler--Shimura isomorphism}\label{EichlerShimura}
The purpose of this section is to show how the periods of $f$ appear "in nature". We will see that from a cohomological point of view, $u_f$ defines the natural higher weight analogue of modular symbols. We will refer to \cite{Wie19} for a comprehensive background. \\
Let $G$ be any group and let $M$ be a left $\Z[G]$-module. Then one can define cohomology groups; 
$$H^i(G,M):=Z^i(G,M)/B^i(G,M), $$ 
consisting of a quotient of certain maps 
$$u:\underbrace{G\times \ldots \times G}_{i}\rightarrow M,$$
corresponding to a specific choice of injective resolution.\\ 
In particular for $i=1$ we have the following explicit description;
\begin{align*} 
Z^1(G,M)= \{  u:G\rightarrow M\mid u(g_1g_2)=u(g_1)+g_1 u(g_2),\, \forall g_1,g_2\in G \},\\
B^1(G,M)=\{  v:G\rightarrow M\mid \exists x_v\in M\text{ such that } v(g)=(g-1)x_v,\, \forall g\in G \}. \end{align*}
Now fix a subset $P\subset G$ and consider  
$$Z^1_P(G,M):= \{ u\in Z^1(G,M)\mid u(p)\in (p-1)M,\forall p\in P\},$$
which we note still contains the boundaries $B^1(G,M)$. From this we define the first $P$-cohomology group as;
$$ H^1_P(G,M):=Z^1_P(G,M)/B^1(G,M).  $$
In our case we consider $G=\Gamma$, a discrete, co-finite, torsion-free subgroup of $\PSL_2(\R)$, and let $P$ be the set of parabolic elements of $\Gamma$. We note that parabolic cohomology groups carry a natural Hecke action.\\ 
Now consider $M=V_{k-2}(\C)\cong \Sym^n(\C^2)$, the space of homogenous polynomials in two variables of degree $k-2$ with coefficients in $\C$, equipped with the following left-action of $\Gamma$;
$$ (\gamma . P)(X,Y):= P((X,Y)\gamma)=P((aX+cY, bX+dY)),  $$
for $\gamma=\begin{psmallmatrix} a&b\\ c & d \end{psmallmatrix}\in \Gamma$ and $P\in V_{k-2}(\C)\subset \C[X,Y]$. From this data we form Eichler's parabolic cohomology group $H^1_P(\Gamma, V_{k-2}(\C))$.\\
Given a cusp form $f\in \mathcal{S}_k(\Gamma)$ of weight $k$, we can define a map $\sigma_f:\Gamma\rightarrow V_{k-2}(\C)$ as;
$$\sigma_f(\gamma)(X,Y):= \int_{\gamma \infty}^\infty f(z)(Xz+Y)^{k-2}dz,$$ 
and it can be shown that $\sigma_f\in Z^1_P(\Gamma,V_{k-2}(\C))$. We similarly define $\sigma_{\overline{f}}\in Z^1_P(\Gamma,V_{k-2}(\C))$ for $\overline{f}\in \overline{\mathcal{S}_k(\Gamma)}$ an anti-holomorphic cusp form of weight $k$. Note that when $k=2$, $\sigma_f$ is exactly {\it the modular symbol map} of \cite[(1.1)]{PeRi}.\\ 
The main theorem of Eichler--Shimura \cite[Proposition 6.2.3,Proposition 6.2.5]{Wie19} is now that the $\C$-linear map;
\begin{align*}
\mathcal{S}_k(\Gamma)\oplus \overline{\mathcal{S}_k(\Gamma)} &\rightarrow H^1_P(\Gamma, V_{k-2}(\C))\\
(f, \overline{g})&\mapsto (\gamma \mapsto \sigma_f(\gamma)+\sigma_{\overline{g}}(\gamma))
\end{align*}
is an isomorphism, which carries a natural action of the Hecke algebra as explained in \cite[Section 8.3]{Sh94} (see also the seminal paper \cite{De71} for a purely algebraic proof of these facts).\\ 

Observe that the periods that we will study in this paper $u_f(\gamma)$ are (up to simple scaling by binomial coefficients) given by the coefficients of $\sigma_f(\gamma)(X,Y)$, and one can use the above to define an (equivalent) action of $\Gamma$ on $\C^{k-1}$ directly (similar to the action of $\Gamma$ on $\R^{k-1}$ described in \cite[Chapter 8]{Sh94}), which was alluded to in the introduction. Thus we see that from a cohomological point of view the periods $u_f$ define a natural generalization of modular symbols when $f\in \mathcal{S}_k(\Gamma)$, $k\geq 4$.\\ 
Furthermore we notice the following obvious connection with the period polynomials defined in (\ref{periodpoldef}); 
$$r_{f,\gamma}(X)=\sigma_f(\gamma)(1,-X).$$
The reason why we used the definition (\ref{periodpoldef}) of the period polynomials was to make the connection to the results listed in \cite{DiRo18} clear.
\begin{remark}In fact there is a notion of {\it modular symbols} associated to $\mathcal{S}_k(\Gamma)$ for all weights $k$ \cite[Section 1.2]{Wie19}, and one can show that the parabolic cohomology groups $H^1_P(\Gamma,V_{k-2}(\C))$ are isomorphic to the {\it cuspidal modular symbols} (see \cite[Theorem 5.2.1]{Wie19} for details)\end{remark}

\subsection{Spectral bounds of sums of Kloosterman sums}
An important ingredient when proving our main results is the cancellation in Kloosterman sums. For arithmetic subgroups we have very strong bounds for individual Kloosterman sums  from Weil's work on the Riemann Hypothesis over finite fields, but for general Fuchsian groups of the first kind, we only have non-trivial bounds when we average over the moduli. Below we will collect the results we will need on Kloosterman sums.\\

Let $\Gamma$ be a co-finite, discrete subgroup of $\PSL_2(\R)$ with a cusp at infinity of width 1. Then we define the Kloosterman sum with frequencies $m,n$ and modulus $c$ (the lower-left entry of some matrix $\gamma \in \Gamma$) as; 
\begin{align}\label{generalgauss} 
S(m,n;c):= \sum_{\begin{pmatrix} a & \ast \\ c & d\end{pmatrix}\in \Gamma_\infty \backslash \Gamma / \Gamma_\infty} e\left(m\frac{d}{c}+n\frac{a}{c}\right).
\end{align}
It can be shown that
$$ \#\left\{ \begin{pmatrix} a & \ast \\ c & d\end{pmatrix}\in \Gamma_\infty \backslash \Gamma / \Gamma_\infty\mid 0\leq c\leq X \right\}\ll X^2, $$
which yields the following trivial bound 
$$ S(m,n;c)\ll c^2,   $$
uniformly in $m,n$, see \cite[Proposition 2.8]{Iw}. If $\Gamma=\Gamma_0(N)$ is a Hecke congruence group, we can do much better by Weil's bound;
\begin{align}\label{deligne}|S(m,n;c)|\leq d(c) c^{1/2} (m,n,c)^{1/2},\end{align}
where $d$ is the divisor function. The point is now that if we average over the moduli $c$, we can also detect cancelation in Kloosterman sums for general $\Gamma$.
 
\subsubsection{Spectral theory of Kloosterman sums}The most powerful tools for obtaining bounds for sums of Kloosterman sums come from the spectral theory of automorphic forms following an approach initiated by Selberg. We refer to \cite{Iw} for a comprehensive background on the spectral theory of automorphic forms.\\ 
In this approach the spectrum of the automorphic Laplacian $\Delta=\Delta_\Gamma$ plays a prominent role, which in local coordinates is given by
$$\Delta_\Gamma=-y^2\left(\frac{\partial^2}{\partial x^2}+\frac{\partial ^2}{\partial y^2} \right).$$ 
It can be shown that $\Delta_\Gamma$ with domain given by smooth and bounded functions on $\Gamma\backslash \H$, defines a non-negative, unbounded operator with a unique self-adjoint extension (which we also denote $\Delta=\Delta_\Gamma$). We observe that $\lambda=0$ is always an eigenvalue of $\Delta_\Gamma$ corresponding to the constant function. Furthermore the famous {\it Selberg conjecture} predicts  that for congruence subgroups $\Gamma_0(N)$ the first non-zero eigenvalue is $\geq 1/4$. It is known that there exist non-congruence subgroups $\Gamma$ such that $\Delta_\Gamma$ has non-zero eigenvalues arbitrarily close to 0 as explained in \cite[(11.15)]{Iw}. \\
For $n=0$ the Kloosterman sums reduce to a generalization of the classical Ramanujan sums and the $m$th Fourier coefficient of the Eisenstein series;
$$ E(z,s)=E_\Gamma(z,s)=\sum_{\gamma \in \Gamma_\infty \backslash \Gamma } \Im(\gamma z)^s$$
is exactly
$$  \Gamma(s)\zeta(2s)^{-1} \sum_{c>0} \frac{S(m,0;c)}{c^{2s}}, $$
where the sum is taken over lower-left entries of matrices in $\Gamma$. Recall that by the general theory of Eisenstein series due to Selberg, $E(z,s)$ has its rightmost pole at $s=1$, which is a simple pole with residue $\vol (\Gamma)^{-1}$, \cite[Proposition 6.13]{Iw}. All the other finitely many poles in $1/2<\Re s< 1$ are also simple and the residues are eigenfunctions for $\Delta$.\\
First of all lets see how to use the analytic properties of Eisenstein series to understand the asymptotic size of the outcome space $\tilde{\Omega}_X$: This is possible since we have a bijection
$$ \Gamma_\infty\backslash \Gamma / \Gamma_\infty\leftrightarrow T_{\leq 1,\Gamma }\cup \{\infty\}, $$
with $T_{\leq 1,\Gamma}$ as defined in (\ref{T}). Thus we see that the constant term in the Fourier expansion of $E(z,s)$ is exactly the generating series for $T_{\leq 1, \Gamma}$. Since the pole of $E(z,s)$ is a constant, the constant term in the Fourier expansion of $E(z,s)$ also has a simple pole (with the same residue). Now by a standard complex analysis argument we get 
\begin{align}\label{T(X)}\# \tilde{\Omega}_X=\frac{X^2}{\vol(\Gamma)}+O(X^{2-\delta_\Gamma }), \end{align}
for some $\delta_\Gamma >0$ depending on the spectral gap for $\Gamma$, with $\tilde{\Omega}_X$ as in (\ref{OmegaX}).\\
Furthermore since the pole at $s=1$ of the Eisenstein series has constant residue, it follows that for $m\neq 0$ the Dirichlet series 
$$\sum_{c} \frac{S(m,0;c)}{c^{2s}},$$
where the sum is over lower-left entries of matrices in $\Gamma$, has analytic continuation to $\Re s>\Re s_1\geq 1/2$ where $\lambda_1=s_1(1-s_1)$ is the smallest non-zero eigenvalue. From this one easily proves 
$$  \sum_{c\leq X} S(m,0;c)\ll_\Gamma |m|^{1/2}X^{2-\delta_\Gamma },  $$
for some $\delta_\Gamma >0$ (see \cite[(3.6)]{PeRi}).\\
For $mn\neq 0$ the corresponding Dirichlet series 
$$\sum_{c} \frac{S(m,n;c)}{c^{2s}}, $$  
shows up in the Fourier coefficients of the Poincaré series 
$$P_m(z,s)=\sum_{\gamma\in \Gamma_\infty\backslash \Gamma} e(m\gamma z) (\Im \gamma z)^s,$$
as was brilliantly used by Goldfeld and Sarnak in \cite{GoSa83} to obtain bounds on sums of Kloosterman sums. Using analytic properties of the resolvent of $\Delta_\Gamma$, they show that $P_m(z,s)$ has meromorphic continuation with possible poles only at the spectrum of $\Delta_\Gamma$ and from this they obtain bounds for sums of Kloosterman sums. For our applications the dependence on $m,n$ is essential, but this dependence is not clear from the statement of their theorem \cite[Theorem 2]{GoSa83}. However using \cite[Remark 1]{GoSa83} one can easily adapt their arguments to deduce the bound
\begin{align} \label{generalKlo} \sum_{c\leq X} S(m,n;c)\ll_\Gamma mnX^{2-\delta_\Gamma },  \end{align}
for some $\delta_\Gamma >0$ depending on the spectral gap of $\Gamma$. We will omit the details. \\

\subsection{Additive twists}\label{AT-section}
The idea behind the proofs of the main theorems is to relate the periods of $f\in \mathcal{S}_k(\Gamma)$ to critical values of additive twists of the $L$-function of $f$. The additive twists are defined as 
$$ L(f,r , s):= \sum_{n\geq 1}\frac{a_f(n)e(nr)}{n^s}, $$
where $r\in \R$ and $e(x)=e^{2\pi i x}$ which apriori converges for $\Re s>(k+1)/2$ by Hecke's bound (\ref{hecke}). If $r$ corresponds to a cusp of $\Gamma$ then $L(f,r , s)$ admits analytic continuation by the integral representation;
$$   L(f,r , s)= \frac{(2\pi)^{s}}{\Gamma(s)} \int_0^\infty f(r+iy)y^{s}\frac{dy}{y}.$$
Furthermore if $r=a/c=\gamma \infty $ with 
$$\gamma = \begin{pmatrix} a& b \\ c & d \end{pmatrix}\in \Gamma,$$
the completed $L$-function satisfies the following functional equation;
\begin{align} \nonumber \Lambda(f,a/c ,s)&:=  \Gamma(s) \left(\frac{c}{2\pi }\right)^{s} L(f,a/c ,s)\\
\label{AddFE}&= (-1)^{k/2}  \Lambda(f,-d/c ,k-s),\end{align}
where $-d/c=\overline{r}=\gamma^{-1}\infty$ (see for instance \cite[Section A.3]{KoMiVa02}).\\
The relation between the periods of $f$ and additive twists is given by the following.

\begin{lemma}
Let $l\in \Z_{\geq 0}$ be a non-negative integer, $\gamma\in\Gamma$ and $f$ as above. Then we have
\begin{align}
\label{period}\int_{\gamma \infty}^\infty f(z) z^l dz=\sum_{j=0}^l \binom{l}{j} (a/c)^{l-j} (-2\pi i)^{-j-1} \Gamma(j+1)L(f,a/c ,j+1),
\end{align}
where $a/c=\gamma \infty$.
\end{lemma}
\begin{proof}
By a straight forward computation we have
\begin{align*}
\int_{\gamma \infty}^\infty f(z) z^l dz&= i \int_{0}^\infty f(a/c+it) (a/c+it)^l dt\\
&= \sum_{j=0}^l i^{j+1}(a/c)^{l-j}\int_{0}^\infty f(a/c+it) t^j dt\\
&= \sum_{j=0}^l i^{j+1}(a/c)^{l-j}(2\pi)^{-j-1}\Gamma(j+1)L(f,a/c ,j+1),
\end{align*}
as wanted.
\end{proof}

It turns out that the dominating term for all of these periods will be the left-most critical value $L(f,a/c ,1)$. This is hinted to by the following proposition.

\begin{prop}\label{bounds}
For $a/c \in T_{\leq 1, \Gamma}$ (i.e. $a,c$ are respectively, the left upper and left lower entries of some matrix in $\Gamma$ such that $0<a/c<1$), we have  the following bounds;
\begin{enumerate}[(i)]
\item $L(f,a/c , \sigma )\ll 1$ for $\sigma\geq k/2+1$,
\item $L(f,a/c , k/2)\ll_\eps c^\eps$,
\item $L(f,a/c , \sigma )\ll c^{k-2\sigma}$ for $\sigma\leq k/2-1$, 
\end{enumerate}
as $c\rightarrow \infty$.
\end{prop}
\begin{proof}
{\bf Case {\it (i)}} For $\sigma\geq k/2+1$ we get by Hecke's bound (\ref{hecke}) the following uniform bound;
$$ L(f,a/c , \sigma) \ll \sum_{n\geq 1}\frac{|a_f(n)|}{n^\sigma}\leq  \sum_{n\geq 1}\frac{|a_f(n)|}{n^{k/2+1}} <\infty, $$
which is independent of $a/c$ and $\sigma\geq k/2+1$.\\

{\bf Case {\it (ii)}} The bound on the central value was proved by the author \cite[Corollary 5.8]{No19}. \\

{\bf Case {\it (iii)}} Finally for $ \sigma\leq k/2-1$, we get by the functional equation (\ref{AddFE}) the following;
$$ L(f,a/c ,\sigma)= \frac{(-1)^{k/2}\Gamma(k-\sigma) (2\pi )^{-k+\sigma}}{\Gamma(\sigma) (2\pi )^{-\sigma}} c^{k-2\sigma} L(f,-d/c , k-\sigma), $$
and since $k-\sigma\geq k/2+1$ the result follows from {\it (i)}. Observe that we avoid the poles of the $\Gamma$-function in the numerator.  
\end{proof}

\section{On the zeroes of the period polynomials}
In this section we will apply the bounds in Proposition \ref{bounds} to determine the asymptotic behavior of the zeroes of the period polynomials associated to a fixed cusp form as the denominator of the cusp varies.\\
Let $f\in \mathcal{S}_k(\Gamma_0(N))$ be a fixed newform of even weight $k\geq 6$. Consider the period polynomials associated to $f$;
$$r_{f,\gamma}(X)=\int_{\gamma \infty}^{\infty} f(z)(z-X)^{k-2}dz=b_{f,k-2}(\gamma)X^{k-2}+\ldots+ b_{f,0}(\gamma),$$
where $\gamma\in \Gamma$ and 
\begin{align*}b_{f,l}(\gamma)&=(-1)^l\binom{k-2}{l}\int_{\gamma\infty}^\infty f(z)z^{k-2-l}dz\\
&=\sum_{j=0}^{k-2-l} \frac{(-1)^l\binom{k-2}{l}\binom{k-2-l}{j}}{(-2\pi i)^{j+1}} (a/c)^{k-2-l-j}  \Gamma(j+1)L(f,a/c ,j+1)\end{align*} 
are the coefficients of $r_{f,\gamma}$. We have the following bound on the Fourier coefficients of $f$ due to Deligne \cite{DeWeil2};
$$|a_f(n)|\leq d(n)n^{(k-1)/2},$$
where $d$ is the divisor function. This implies that
$$ \sum_{n\geq 2} \frac{|a_f(n)|}{n^{k-1}} \leq \sum_{n\geq 2} \frac{d(n)}{n^{(k-1)/2}}= \zeta((k-1)/2)^2-1\leq  \zeta(5/2)^2-1=0.799...<1.$$
This shows that $ L(f,x ,k-1)$ is bounded both from above and away from zero uniformly in $x\in \R$. Combining this observation with the functional equation for additive twists we conclude that
\begin{align} 
\nonumber b_{f,k-2}(\gamma)&=\frac{L(f,\gamma \infty,1)}{-2\pi i} \\
\label{firstcoef}&= \frac{(-1)^{k/2}i \Gamma(k-1)}{(2\pi)^{k-1}}L(f,\gamma^{-1} \infty ,k-1)c^{k-2}\neq 0. 
\end{align}
Thus $r_{f,\gamma} $ is actually a polynomial of degree $k-2$ and normalizing it so that it becomes a monic polynomial the coefficients become; 
$$\tilde{b}_l(\gamma)=\tilde{b}_{f,l}(\gamma):= b_{f,l}(\gamma)/b_{f,k-2}(\gamma),\quad l=0,\ldots, k-2.$$
We can now prove the promised asymptotic expression for the zeroes of $r_{f, \gamma}$ as $c\rightarrow \infty$.
 \begin{proof}[Proof of Theorem \ref{zeroes}]
 Let $r=\gamma\infty=a/c$. Using (\ref{firstcoef}) we see that $b_{f,k-2}(\gamma)\gg_k c^{k-2}$. Combining this with the expression (\ref{period}) and the bounds from Proposition \ref{bounds}, we conclude the following;
\begin{align}\nonumber \tilde{b}_l(\gamma)&=(-1)^l\binom{k-2}{l} r\, ^{k-2-l}+O_k\left(\sum_{j=1}^{k-2-l}|r|^{k-2-l-j} \frac{|L(f,r,j+1)|}{|b_{f,k-2}(\gamma)|}\right)\\
\nonumber &=(-1)^l\binom{k-2}{l} r\, ^{k-2-l}+O_k\left(\sum_{j=1}^{k-2-l}|r|^{k-2-l-j} c^{\max(0,k-2j-2)}c^{-(k-2)}\right).\end{align}
One easily checks that $|r|^{k-2-l-j} c^{\max(0,k-2j-2)}c^{-(k-2)} \ll |r|^{k-4-l}c^{-2}$ for all $j=1,\ldots, k-2-l$ using that $|r|\geq c^{-1}$. Thus we conclude
\begin{align}\label{coef}  \tilde{b}_l(\gamma)=  (-1)^l\binom{k-2}{l} r\, ^{k-2-l}+O_k(|r|^{k-4-l}c^{-2}),\end{align}
and in particular $\tilde{b}_l(\gamma)\ll_k |r|^{k-2-l}$.\\ 
Now we will show that any zero $x_0$ of $r_{f, \gamma}$ is bounded by $O_k(| r |)$. So assume that a zero $x_0$ of $r_{f,\gamma}$ satisfies $|x_0|\geq |r|$. Then using (\ref{coef}), we get the bound
 $$|x_0|^{k-2}=|-\tilde{b}_{k-3}(\gamma)x_0^{k-3}-\ldots -\tilde{b}_0(\gamma)|\ll_k |r| |x_0|^{k-3},  $$
 which implies $x_0\ll_k |r|$ as wanted. \\
Now combining $x_0\ll_k |r|$ with (\ref{coef}), we conclude for any root $x_0$ of $r_{f,\gamma}$ we have that
 $$0=x_0^{k-2}+\tilde{b}_{k-3}(\gamma)x_0^{k-3}+\ldots +\tilde{b}_0(\gamma)= (x_0-r)^{k-2}+O_k((|r|+1)^{k-4}c^{-2}), $$
 which implies that $|x_0-r|\ll_k (|r|+1)^{(k-4)/(k-2)}c^{-2/(k-2)} $ as wanted.
 \end{proof}
If we restrict to $\gamma \in \Gamma_\infty\backslash\Gamma_0(N)$ such that $r=a/c=\gamma \infty \in \Omega_c$ (i.e. $0<r=a/c<1$) we conclude that the zeroes of $r_{f,\gamma}$ satisfy
$$x_0=a/c+O_k(c^{-2/(k-2)}). $$


\section{On the distribution of the Eichler--Shimura map} 
In this section we will prove Theorem \ref{distN} and Theorem \ref{distN2} using the method of moments. More precisely this is done by firstly computing all the moments of the random variable $u_f$ on respectively $\Omega_c$ and $\tilde{\Omega}_X$ and then applying a result from probability theory due to Fréchet--Shohat to determine the limiting distribution. 
\subsection{Computation of the moments of $u_f$}
To state our results we let (as above)
$$ f(z)=\sum_{n\geq 1} a_f(n)  q^n,$$
be the Fourier expansion of a cusp form $f\in \mathcal{S}_k(\Gamma)$. Then we define the following Dirichlet series for $\alpha,\beta\in \Z_{\geq 0}$;
\begin{align} \label{momentsL} L_{f, \alpha, \beta}(s):&= \sum_{\substack{n_1,\ldots , n_{\alpha+\beta}>0\\ n_1+\ldots+n_\alpha=n_{\alpha+1}+\ldots +n_{\alpha+\beta}}} \frac{a_f(n_1)\cdots a_f (n_{\alpha})\overline{a_f (n_{\alpha+1})}\cdots \overline{a_f (n_{\alpha+\beta})}}{(n_1\cdots n_{\alpha+\beta})^s}\\
\nonumber &=\int_0^1 L(f,x , s)^{\alpha,\beta}dx, \end{align}
which converges absolutely for $\Re s>(k+1)/2$ by Hecke's bound (\ref{hecke}), where we use the notation $z^{\alpha,\beta}=z^\alpha \overline{z}^\beta$.\\
For $\Gamma=\Gamma_0(N)$ a Hecke congruence group, we get the following calculation of the moments.  
\begin{thm} \label{moment1}
Let $ f\in \mathcal{S}_k(\Gamma_0(N))$ be a cusp form of even weight $k\geq 4$. Then for any non-negative integers;
$$\alpha_0,\ldots, \alpha_{k-2}, \beta_0,\ldots, \beta_{k-2},$$ 
not all zero and $c\equiv 0\, (N)$, we have that
\begin{align} \nonumber & \frac{1}{\varphi(c)}\sum_{\substack{0\leq a<c,\\ (a,c)=1}} \prod_{j=0}^{k-2}  \left (\frac{(2\pi/c)^{k-2}}{\Gamma(k-1)i}\int_{a/c}^\infty f(z) z^{j}dz\right)^{\alpha_j,\beta_j}\\
&= \frac{L_{f, \alpha, \beta} (k-1)}{1+\sum_{j=0}^{k-2}j\cdot (\alpha_j+\beta_j)} +O_{\eps,\alpha,\beta,f} (c^{-1/6+\eps}),\end{align}   
where $\alpha=\alpha_0+\ldots+\alpha_{k-2}$ and $\beta=\beta_0+\ldots+\beta_{k-2}$.
\end{thm} 
For a general cofinite Fuchsian group $\Gamma$, we have to take an extra average in order to calculate the moments.
\begin{thm}\label{moment2}
Let $\Gamma$ be a cofinite Fuchsian group with a cusp at $\infty$ of width 1 and let $ f\in \mathcal{S}_k(\Gamma)$ be a cusp form of even weight $k\geq 4$. Then for any non-negative integers;
$$\alpha_0,\ldots, \alpha_{k-2}, \beta_0,\ldots, \beta_{k-2},$$ 
 not all zero, we have that
\begin{align} \nonumber &\frac{1}{\#\tilde{\Omega}_X}\sum_{r \in  \tilde{\Omega}_X} \prod_{j=0}^{k-2}  \left (\frac{(2\pi/c(r))^{k-2}}{\Gamma(k-1)i}\int_r^\infty f(z) z^{j}dz\right)^{\alpha_j,\beta_j} 
\\ &=\frac{L_{f, \alpha, \beta} (k-1)}{1+\sum_{j=0}^{k-2}j\cdot (\alpha_j+\beta_j)}+O_{\alpha,\beta} (X^{-\delta_\Gamma }),\end{align}   
for some $\delta_\Gamma  >0$ depending on the spectral gap of $\Gamma$, where $\alpha=\alpha_0+\ldots+\alpha_{k-2}$ and $\beta=\beta_0+\ldots+\beta_{k-2}$.
\end{thm}
\begin{remark}\label{momentsmatch}Observe that the main terms above are exactly what we expect from the statements of Theorem \ref{distN} and Theorem \ref{distN2}, since $L_{f, \alpha, \beta}(k-1)$ is precisely the $(\alpha, \beta)$-moment of $L(f,Y ,k-1)$ with $Y$ equidistributed on $[0,1)$  (see (\ref{momentsL})) and 
$$  \int_{0}^1 z^{\alpha_1+\beta_1}z^{2(\alpha_2+\beta_2)}\cdots z^{(k-2)(\alpha_{k-2}+\beta_{k-2})}dz= \frac{1}{1+\sum_{j=0}^{k-2}j\cdot (\alpha_j+\beta_j)}. $$  
\end{remark}
\begin{proof}[Proof of Theorem \ref{moment1} and Theorem \ref{moment2}] In the following all implied constants may depend on $f$, $\alpha$ and $\beta$. In view of (\ref{period}) we can express the periods of $f$ as a linear combination of critical values of the additive twists $L(f,r ,s)$ and by the functional equation, we have the equality
$$ L(f,r , 1)= c(r)^{k-2}\frac{\Gamma(k-1)}{(2\pi)^{k-2}} L(f,\overline{r},k-1)$$
with $r=\gamma \infty$ and $\overline{r}=\gamma^{-1}\infty$. Using Proposition \ref{bounds} this implies that 
\begin{align}
\label{alphabeta} &\prod_{j=0}^{k-2}  \left (\frac{(2\pi/c(r))^{k-2}}{\Gamma(k-1)i}\int_r^\infty f(z) z^{j}dz\right)^{\alpha_j,\beta_j}\\
\nonumber =&L(f,\overline{r},k-1)^{\alpha,\beta}  r^{M}+O(c(r)^{-2})  
\end{align}
where $z^{\alpha,\beta}=z^\alpha \overline{z}^\beta$ and
$$M=M(\alpha_1,\ldots , \alpha_{k-2},\beta_1,\ldots , \beta_{k-2}):=\sum_{j=0}^{k-2}j\cdot (\alpha_j+\beta_j).$$  \\ 
In order to deal with the term $r^M$, we apply a standard smooth approximation. So let $\varphi: \R\rightarrow \R_{\geq 0}$ be a smooth function with compact support in $(0,1)$ such that $\int_0^1 \varphi(x)dx=1$. Then we define the following approximation to the Dirac measure at $x=0$;
$$  \varphi_\delta (x):= \delta^{-1} \varphi(x/\delta), $$
where $\delta>0$ is some small constant to be chosen appropriately. Notice that $\varphi_\delta$ is supported in $(0,\delta)$ and satisfies $\int_\R\varphi_\delta(x)dx=1$. We think of $\varphi_\delta$ as a function on the circle $\R/\Z$ by extending its values on $[0,1)$ periodically.\\
Associated to the periodic functions $h_j: \R/\Z \rightarrow \R$ given by $h_j(x)=x^j$ for $x\in [0,1)$ with $j\in \Z_{\geq 0}$, we define the following smooth approximation;
$$h_{j,\delta}:=h_j\ast \varphi_\delta, $$ 
where $\ast$ denotes the (additive) convolution product on $\R/\Z$. The approximation $h_{j,\delta}$ satisfies the following standard properties (which can be proved easily by partial integration using the compact support of $\varphi_\delta$); 
\begin{align}\label{fourierbound1}\widehat{h_{j,\delta}}(l) \ll _A \frac{1}{(\delta(1+|l|))^A}, \quad \widehat{h_{j,\delta}}(0)=\widehat{h_j}(0)=\frac{1}{j+1},\end{align}
where $A>0$ and $\widehat{h_{j,\delta}}$ denotes the Fourier transform on $\R/\Z$. And furthermore
 $$h_{j,\delta} (x)= h_j(x)+\int_{x-\delta}^x (h_j(t)-h_j(x))\varphi_\delta(x-t)dt = h_j(x)+O_j(\delta),$$
for $\delta \leq x <1$ using that $h_j$ is differentiable. This estimate fails for $0\leq x\leq \delta$, but it is standard to show that the contribution from $r\in \tilde{\Omega}_X $ (respectively $r\in \Omega_c$) with $r<\delta$ is negligible and will not affect the error terms. More precisely it is obvious that $\{r\in \Omega_c \mid r<\delta \}\ll \delta c $ and by using the cancellation in Kloosterman sums one can easily show that $\{r\in \tilde{\Omega}_X \mid r<\delta \}\ll \delta X^2$.\\ 
The upshot is that we can replace $r^M$ by the approximation $h_{M,\delta}$ at the cost of changing the error term in (\ref{alphabeta}) to $O(\delta+c(r)^{-2})$ (at least when we average over $\Omega_c$, respectively $\tilde{\Omega}_X$).\\

Finally we replace $h_{M,\delta}$ by its Fourier expansion to arrive at the following expression for the main term;
\begin{align}
\nonumber &L(f,\overline{r},k-1)^{\alpha,\beta}  h_{M,\delta}(r)\\
\nonumber= &\sum_{l\in \Z} \widehat{h_{M, \delta}}(l) e(lr) L(f,\overline{r},k-1)^{\alpha,\beta}   \\
\nonumber = &\sum_{l\in \Z} \widehat{h_{M,\delta}}(l) \sum_{n_1,\ldots, n_{\alpha+\beta}>0} \frac{a_f(n_1)\cdots a_f(n_\alpha) \overline{a_f(n_{\alpha+1})}\cdots\overline{a_f(n_{\alpha+\beta})}}{(n_1\cdots n_{\alpha+\beta})^{k-1}}\\
 \label{rawsum}&\times e(lr+\overline{r}(n_1+\ldots+n_\alpha -n_{\alpha+1}-\ldots -n_{\alpha+\beta})),
\end{align}	
using that $L(f,\overline{r},k-1)$ is absolutely convergent and so is the Fourier expansion of $h_{M,\delta}$ in view of (\ref{fourierbound1}).\\

Now the case where $\Gamma=\Gamma_0(N)$ is a Hecke congruence group, we average (\ref{alphabeta}) over $r\in \Omega_c$. Since all of the $r$-dependence is in the exponential, we see the Kloosterman sums entering the picture. The main contribution comes from the {\it diagonal terms} corresponding to $l=0$ and $n_1+\ldots+n_\alpha =n_{\alpha+1}+\ldots +n_{\alpha+\beta}$, which contribute
\begin{align}\label{maincon} L_{f,\alpha,\beta}(k-1)\widehat{h_{M,\delta}}(0)=L_{f,\alpha,\beta}(k-1) \frac{1}{M+1}.\end{align}
In order to handle the off-diagonal contributions, we apply Weil's bound (\ref{deligne}), which bounds the off-diagonal terms by the following; 
\begin{align*}&\ll \frac{d(c)c^{1/2}}{\varphi(c)} \sum_{l\neq 0} \sum_{n_1,\ldots, n_{\alpha+\beta}} |\widehat{h_{M,\delta}}(l)|\frac{|a_f(n_1)\cdots a_f(n_{\alpha+\beta})|}{(n_1\cdots n_{\alpha+\beta})^{k-1}}\left(l,c, \sum_{i=1}^\alpha n_i-\sum_{j=\alpha+1}^{\alpha+\beta} n_j\right)^{1/2}\\ 
&\ll_{\eps,\alpha,\beta} \frac{c^{1/2+\eps}}{\varphi(c)} \left(\sum_{l\neq 0} |\widehat{h_{M,\delta}}(l)| \right) \left( \sum_{n_1,\ldots, n_{\alpha+\beta}} \frac{|a_f(n_1)\cdots a_f(n_{\alpha+\beta})|}{(n_1\cdots n_{\alpha+\beta})^{k-1}} \max(n_1,\ldots, n_{\alpha+\beta}  )^{1/2-\eps}\right)\\
&\ll_{\eps,\alpha,\beta} \frac{c^{1/2+\eps}}{\varphi(c)} \sum_{l\neq 0} |\widehat{h_{M,\delta}}(l)|,
\end{align*} 
using Hecke's bound (\ref{hecke}) to show finiteness of the sum over $n_1,\ldots, n_{\alpha+\beta}$. Combining the above with the fact that $\varphi(c)\gg_\eps c^{1-\eps}$, we arrive at the following;
\begin{align}
\nonumber &\frac{1}{\varphi(c)}\sum_{\substack{0\leq a<c\\ (a,c)=1}} \prod_{j=0}^{k-2}  \left (\frac{(2\pi/c)^{k-2}}{i\Gamma(k-1)}\int_{a/c}^\infty f(z) z^{j}dz\right)^{\alpha_j,\beta_j}\\
 =  &L_{f,\alpha,\beta}(k-1)\widehat{h_{M,\delta}}(0)+O_\eps\left(\delta+c^{-2}+ c^{-1/2+\eps} \sum_{l\neq 0}| \widehat{h_{M,\delta}}(l)|\right).
\end{align}
Next we apply (\ref{fourierbound1}) with $A=2+\eps$ to ensure convergence of the sum $\sum_{l\neq 0}| \widehat{h_{M,\delta}}(l)|$ and arrive at the following error term $O_\eps(\delta+c^{-2}+ c^{-1/2+\eps}\delta^{-2-\eps})$. Finally we choose $\delta=c^{-1/6}$ to balance the error terms.\\

The argument when $\Gamma$ is a general cofinite Fuchsian group is similar, only now we average (\ref{alphabeta}) over $r\in \tilde{\Omega}_X$. In this case we also see Gauss sums enter the picture since we have
$$ \sum_{r\in \tilde{\Omega}_X} e(nr+m\overline{r})= \sum_{0<c\leq X} S_\Gamma(m,n;c),$$
where the sum is taking over lower left entries $c$ of matrices in $\Gamma$ and $S_\Gamma(m,n;c)$ is a (generalized) Kloosterman sum defined by (\ref{generalgauss}).\\
 
Again the main contribution is given by (\ref{maincon}). When dealing with the off-diagonal contribution, we first of all have to trivially bound the terms in (\ref{rawsum}) with $\min (n_1,\ldots, n_{\alpha+\beta})> X^{\delta_1}$ for some $\delta_1>0$ to be chosen appropriately. This is necessary since the dependence on the frequencies in (\ref{generalKlo}) is not as strong as in Weil's bound (actually this extra step is only needed when $k=4$).\\ 
Now using the trivial bound for the exponentials, this truncation yields
\begin{align}
\nonumber &\frac{1}{\# \tilde{\Omega}_X} \sum_{r\in \tilde{\Omega}_X} L(f,\overline{r},k-1)^{\alpha,\beta}  h_{M,\delta}(r) \\
\nonumber = &\sum_{l\in \Z} \widehat{h_{M,\delta}}(l) \sum_{0<n_1,\ldots, n_{\alpha+\beta}<X^{\delta_1}} \frac{a_f(n_1)\cdots a_f(n_\alpha) \overline{a_f(n_{\alpha+1})}\cdots\overline{a_f(n_{\alpha+\beta})}}{(n_1\cdots n_{\alpha+\beta})^{k-1}}\\
\label{rawsum2}&\times \frac{1}{\# \tilde{\Omega}_X}\sum_{r\in \tilde{\Omega}_X}e(lr+\overline{r}(n_1+\ldots+n_\alpha -n_{\alpha+1}-\ldots -n_{\alpha+\beta}))+O(X^{-\delta_1(k-3)/2}).
\end{align}
Now we apply the bound for sums of Kloosterman sums (\ref{generalKlo}) which yields the following bound for the remaining off-diagonal contribution from (\ref{rawsum2});
\begin{align*}
&\ll_{\alpha,\beta}  X^{-\delta_\Gamma } \left( \sum_{l}|\widehat{h_{M,\delta}}(l)| \cdot |l| \right)\\
&\qquad \qquad \qquad\qquad \qquad \times \sum_{0<n_1,\ldots, n_{\alpha+\beta}<X^{\delta_1}} \frac{|a_f(n_1)\cdots a_f(n_{\alpha+\beta})|}{(n_1\cdots n_{\alpha+\beta})^{k-1}} \max(n_1,\ldots, n_{\alpha+\beta}) \\
&\ll_{\alpha, \beta} X^{-\delta_\Gamma } \left( \sum_{l}|\widehat{h_{M,\delta}}(l)| \cdot |l| \right) \max (1, X^{-\delta_1 (k-5)/2})
   , \end{align*}
using also that $\# \tilde{\Omega}_X\gg X^2$ by (\ref{T(X)}).\\
Now we apply (\ref{fourierbound1}) with $A=3+\eps$ to ensure finiteness of the first sum above and then choose $\delta$ and $\delta_1$ to balance the error terms. This yields a power savings, which we will not make explicit. This finishes this case as well.
 \end{proof}

\subsection{Determining the limiting distribution} 
In order to conclude the proofs of Theorem \ref{distN} and Theorem \ref{distN2}, we need to setup our problem in a probability theoretical framework.\\
Let $f\in \mathcal{S}_k(\Gamma)$ be as above and consider the following normalization of the periods of $f$; 
$$  \tilde{u}_{f,j}(r):= \frac{(2\pi /c(r))^{k-2}}{i\Gamma(k-1)}u_{f,j}(r)=\frac{(2\pi /c(r))^{k-2}}{i\Gamma(k-1)} \int_r^\infty f(z) z^j dz , \quad j=0,\ldots, k-2,$$
where $r=\gamma \infty$ with $\gamma\in \Gamma$.
According to whether $\Gamma$ is a congruence subgroup or not, we consider for each $c\equiv 0\, (N)$ (respectively $X>0$) the renormalized periods;
$$ \tilde{u}_f:=(\tilde{u}_{f,0},\ldots , \tilde{u}_{f,k-2} ),$$
as random variables defined on the outcome space $\Omega_c$ (respectively $\tilde{\Omega}_X$) endowed with the discrete $\sigma$-algebra and the uniform probability measure. Then one can easily check as in Remark \ref{momentsmatch} that Theorem \ref{moment1} (respectively Theorem \ref{moment2}) implies that as $c\rightarrow \infty$ (respectively $X\rightarrow \infty$), the moments of the random variables $\tilde{u}_f$ converge to those of the random variable
$$F(Y,Z)=(F_0(Y,Z),\ldots, F_{k-2}(Y,Z))^T,$$ 
where $Y,Z$ are two independent random variables uniformly distributed with respect to the Lebesgue measure on $[0,1)$ and 
$F:[0,1)\times [0,1)\rightarrow \C^{k-1}$ is given (as in Theorem \ref{distN}) by 
$$F(y,z)= L(f,y ,k-1)(1, z, \ldots, z^{k-2})^T\in \C^{k-1}.$$ 
As an example let us consider the (complex) moment of $F(Y,Z)$ corresponding to the tuple $((1,1),(1,1),\ldots, (1,1))$; 
\begin{align*}& \mathbb{E}\left(F(Y,Z)^{((1,1),(1,1),\ldots, (1,1))}\right):=\int_0^1 \int_0^1 |F_0(y,z)|^2\cdots |F_{k-2}(y,z)|^2dydz\\
&=\left(\int_0^1 |L(f,y,k-1)|^{2(k-1)} dy\right)\left( \int_0^1 z^{0+2+\ldots+2(k-2)}dz\right)\\
&= \left(\sum_{\substack{n_1+\ldots +n_{k-1}\\=n_{k}+\ldots +n_{2(k-1)}}} \frac{a_f(n_1)\ldots \overline{a_f(n_{2(k-1)})}}{(n_1\cdots n_{2(k-1)})^{k-1}}\right)\cdot \frac{1}{1+(k-2)(k-1)}, \\
\end{align*}
which we see match the corresponding moment in Theorem \ref{moment1} and Theorem \ref{moment2}.\\

In order to conclude that the random variables associated with $\tilde{u}_f$ converge in distribution to $F(Y,Z)$ as $c\rightarrow \infty$ (respectively $X\rightarrow \infty$), we will combine three results from probability theory due to Fréchet--Shohat, Cramér--Wold and Carleman respectively. A similar but slightly simpler argument was carried out in \cite[Section 5.4]{No19}. 
\begin{proof}[Proof of Theorem \ref{distN} and Theorem \ref{distN2}]
Given a sequence of 1-dimensional real random variables $(X'_n)_{n\geq 1}$ such that all moments exist and converge as $n\rightarrow \infty$ to the moments of some other random variable $Y'$ then it follows from the Fréchet--Shohat Theorem \cite[p. 17]{Serf} that if $Y'$ is uniquely determined by its moments then the random variables $(X'_n)_{n\geq 1}$ converge {\it in distribution} to $Y'$.\\ 
Our random variables are however multidimensional so we have to combine the Fréchet--Shohat Theorem with a result of Cramér and Wold \cite[p. 18]{Serf}, which says that if $(X'_n)_{n\geq 1}$ is a sequence of $(d+1)$-dimensional real random variables;
$$X'_{n}=(X'_{n,0},\ldots , X'_{n,d}),$$
and $Y'=(Y'_0,\ldots , Y'_d)$ is a $(d+1)$-dimensional random variable such that
$$ t_0X'_{n,0}+\ldots +t_d X'_{n,d} $$
converge in distribution as $n\rightarrow \infty$ to
$$ t_0Y'_0+\ldots +t_d Y'_d $$
for any $(d+1)$-tuple $(t_0,\ldots, t_d)\in \R^{d+1}$, {\it then} $X'_n$ converges in distribution to $Y'$ as $n\rightarrow \infty$.\\
Thus by combining Fréchet--Shohat and Cramér--Wold with our calculation of the moments in Theorem \ref{moment1} (respectively Theorem \ref{moment2}), it is enough to show that for any (say non-trivial) linear combination, the following random variable;
\begin{align}\label{FrechetShohat}&t_0  \Re L(f,Y,k-1) +t_1  \Re L(f,Y,k-1)Z+  \\
\nonumber& \qquad \qquad \qquad \qquad \qquad \qquad \qquad \qquad\ldots+t_{k-2} \Re L(f,Y,k-1)Z^{k-2}\\
\nonumber &+t_{k-1}  \Im L(f,Y ,k-1) +t_{k}  \Im L(f,Y ,k-1)Z+  \\
\nonumber& \qquad \qquad \qquad \qquad \qquad \qquad \qquad \qquad\ldots+t_{2k-3} \Im L(f,Y ,k-1)Z^{k-2}
\end{align}
is uniquely determined by its moments. By a condition due to Carleman (see (\ref{carleman}) below), this boils down to showing that the moments are sufficiently bounded from above, which is clear in our case since $Z$ is bounded by 1 and 
$$|L(f,Y ,k-1)|\leq \sum_{n\geq 1} \frac{|a_f(n)|}{n^{k-1}}<\infty,$$ 
both with probability one. To sum up and be precise; if we denote by $\alpha_{2m}$ the $2m$'th moment of (\ref{FrechetShohat}), then we have 
\begin{align}
\label{carleman}\sum_{m\geq 1} \alpha_{2m}^{-1/2m}\geq \sum_{m\geq 1} \left(\left(c(t_0,\ldots, t_{2k-3}) \sum_{n\geq 1} \frac{|a_f(n)|}{n^{k-1}}\right)^{2m}\right)^{-1/2m}=\infty,
 \end{align}
where $c(t_0,\ldots, t_{2k-3})$ is a certain constant depending on $t_0,\ldots, t_{2k-3}$. Thus it follows from the Carleman condition \cite[p. 46]{Serf} that the random variable (\ref{FrechetShohat}) is uniquely determined by its moments. Thus we conclude the proof of Theorem \ref{distN} and Theorem \ref{distN2} using the results of Fréchet--Shohat and Cramér--Wold mentioned above. 
\end{proof}

\bibliographystyle{amsplain}

\providecommand{\bysame}{\leavevmode\hbox to3em{\hrulefill}\thinspace}
\providecommand{\MR}{\relax\ifhmode\unskip\space\fi MR }
\providecommand{\MRhref}[2]{%
  \href{http://www.ams.org/mathscinet-getitem?mr=#1}{#2}
}
\providecommand{\href}[2]{#2}

\end{document}